\documentclass[british,english]{article}
\usepackage[T1]{fontenc}
\usepackage[latin9]{inputenc}
\usepackage{color}
\usepackage{amsmath}
\usepackage{amssymb}
\usepackage{esint}
\usepackage{babel}
\usepackage{bbm}
\allowdisplaybreaks
\DeclareMathOperator{\sign}{sgn}
\begin{document}
\selectlanguage{british}%
\global\long\def\intr{\int_{R}}
 \global\long\def\sbr#1{\left[ #1\right] }
 \global\long\def\cbr#1{\left\{  #1\right\}  }
 \global\long\def\rbr#1{\left(#1\right)}
 \global\long\def\ev#1{\mathbb{E}{#1}}
 \global\long\def\E{\mathbb{E}}
 \global\long\def\P{\mathbb{P}}
 \global\long\def\R{\mathbb{R}}
 \global\long\def\norm#1#2#3{\Vert#1\Vert_{#2}^{#3}}
 \global\long\def\pr#1{\mathbb{P}\rbr{#1}}
 \global\long\def\cleq{\lesssim}
 \global\long\def\ceq{\eqsim}
 \global\long\def\conv{\rightarrow}
 \global\long\def\Var#1{\text{Var}(#1)}
 \global\long\def\TDD{}
 \global\long\def\dd#1{\textnormal{d}#1}
 \global\long\def\inti{\int_{0}^{\infty}}
 \global\long\def\crr{\mathcal{C}([0;\infty),\R)}
 \global\long\def\sb#1{\langle#1\rangle}
 \global\long\def\pm#1{d_{P}\rbr{#1}}
 \global\long\def\crt{\mathcal{C}([0;T],\R)}
 \global\long\def\nuu{\nu_{n;\lambda}}
 \global\long\def\ZZ{Z_{\Lambda_{n}}}
 \global\long\def\PP{\mathbb{P}_{\Lambda_{n}}}
 \global\long\def\EE{\mathbb{E}_{\Lambda_{n}}}
 \global\long\def\LL{\Lambda_{n}}
 \global\long\def\AA{\mathcal{A}}
 \global\long\def\evx{\mathbb{E}_{x}}
 \global\long\def\pin#1{1_{\cbr{#1\in\mathcal{A}}}}
 \global\long\def\Zd{\mathbb{Z}^{d}}
 \global\long\def\ddp#1#2{\langle#1,#2\rangle}
 \global\long\def\intc#1{\int_{0}^{#1}}
 \global\long\def\T#1{\mathcal{P}_{#1}}
 \global\long\def\ii{\mathbf{i}}
 \global\long\def\star#1{\left.#1^{*}\right.}
 \global\long\def\pspace{\mathcal{C}}
 \global\long\def\eq{\varphi}
 \global\long\def\grad{\text{grad}}
 \global\long\def\var{\text{var}}
 \global\long\def\ab{[a,b]}
 \global\long\def\ra{\rightarrow}
 \global\long\def\TTV#1#2#3{\text{TV}^{#3}\!\rbr{#1,#2}}
 \global\long\def\V#1#2#3{\text{V}^{#3}\!\rbr{#1,#2}}
 \global\long\def\vi#1#2{\text{vi}\rbr{#1,#2}}
 \global\long\def\eqdef{:=}
 \global\long\def\UTV#1#2#3{\text{UTV}^{#3}\!\rbr{#1,#2}}
 \global\long\def\DTV#1#2#3{\text{DTV}^{#3}\!\rbr{#1,#2}}
 \global\long\def\ns{\infty}
 \global\long\def\f{:\left[a,b\right]\ra\R}
 \global\long\def\TV{\text{TV}}
 \global\long\def\osc{\text{osc}}
 \global\long\def\calu{{\cal U}^{p}\left[a,b\right]}
 \global\long\def\cont{\text{cont}}
\global\long\def\barP{\mathbb{\bar{P}}}

\newtheorem{theorem}{Theorem} 
\newtheorem{prop}[theorem]{Proposition} 
\newtheorem{lem}[theorem]{Lemma}
\newtheorem{rem}[theorem]{Remark} 
\newtheorem{defi}[theorem]{Definition}
\newtheorem{coro}[theorem]{Corollary}

\newenvironment{proof}{\par\noindent{\bf Proof.}}{\par\rightline{$\blacksquare$}}

\title{On tails of symmetric and totally asymmetric $\alpha$-stable distributions }

\author{Witold M. Bednorz\footnote{University of Warsaw}, Rafa{\l} M. \L ochowski\footnote{Warsaw School of Economics and the University of Warsaw}  and Rafa{\l} Martynek \footnote{University of Warsaw} \footnote{The research of all authors was funded by the National Science Centre, Poland, under Grant
No. 2016/21/B/ST1/0148.}}

\date{}
\maketitle

\abstract{We estimate  up
to universal constants tails of symmetric and totally asymmetric 1-dimensional $\alpha$-stable distributions in terms of functions of the parameters of these distributions. In particular, for values of $\alpha$ close to $2$ we specify where exactly the tail changes from being Gaussian and starts to behave like in the Pareto distribution}

\section{Introduction}

A random variable $X$ is called (one-dimensional) stable if for any numbers $a,b>0$ and $X_{1}, X_{2}$\textendash  independent copies of $X$ there exist numbers $c(a,b)$ and $d(a,b)$ such that $$aX_{1}+bX_{2}\overset{d}{=}c(a,b)X+d(a,b).$$ Random variables of this type constitute an important family used in stochastic modelling. Let us recall some fundamental properties of stable distributions. For the comprehensive study see e.g. \cite{SamorodnitskyTaqqu:1994}. It is a classic result that $c(a,b)$ is of the form $(a^{\alpha}+b^{\alpha})^{\frac{1}{\alpha}}$ for $\alpha\in(0,2]$. Number $\alpha$ is sometimes called index of stability and a stable random variable with index $\alpha$ is called $\alpha$-stable. For $\alpha\neq1$ the characteristic function of $X$ is given by
$$\E\exp(itX)=\exp\left(-\sigma^{\alpha}|t|^{\alpha}\left( 1-i\beta\sign{(t)}\tan\left(\frac{\pi\alpha}{2}\right)\right)+i\mu t\right),$$
while for $\alpha=1$ the characteristic function is given by
$$\E\exp(itX)=\exp\left(-\sigma|t|\left( 1+i\beta\frac{2}{\pi}\sign{(t)}\log(|t|)\right)+i\mu t\right),$$
where $\sigma>0, \beta\in[-1,1]$ and $\mu\in\R$ is a shift parameter. $\beta$ is a skewness (asymmetry) parameter, while $\sigma$ is a scale parameter. The case $\beta=0$ refers to the symmetric case and $\beta=-1$ or $\beta=1$ refer to totally asymmetric case. 
When $\mu=0$ we call $X$ a strictly $\alpha$-stable random variable, in which case the characteristic function can be represented as 
\begin{equation}\label{char}
\E\exp(itX)=\exp\left(\int_{\R}\psi(t,x)\nu(dx)\right),
\end{equation}
where $$\psi(t,x)=\begin{cases}e^{itx}-1 \quad &\text{if}\quad \alpha\in(0,1),\\
  e^{itx} -1-itx \quad &\text{if}\quad \alpha\in(1,2) 
\end{cases}$$ and $\nu$ is called a L\'evy measure given by $$\nu(dx)=\frac{C_{1}}{x^{\alpha+1}}\mathbbm{1}_{(0,\infty)}(x)dx+\frac{C_{2}}{|x|^{\alpha+1}}\mathbbm{1}_{(-\infty,0)}(x)dx,$$
where $C_1,C_2\geq0$ and $C_1+C_2>0$. The relation between $C_1, C_2$ and $\beta$ is given by the equation $\beta=\frac{C_{1}-C_{2}}{C_{1}+C_{2}}$. In particular, for the symmetric case we take $C_1=C_2=1$ and for totally asymmetric case $C_1=1$ and $C_2=0$. Moreover, a dependence of the scale parameter $\sigma$ on the parameter $\alpha$ and constants $C_1$ and $C_2$ is given by $\sigma^{\alpha}=\Gamma(-\alpha)\cos\left( \frac{(2-\alpha)\pi}{2} \right)(C_1+C_2),$ where $\Gamma$ denotes the gamma function. 

\noindent 
There are usually no closed formulas for densities and distribution functions of stable distributions. The exception being the case of Gaussian distribution ($\alpha=2, \beta=0$), Cauchy distribution ($\alpha=1$, $\beta=0$) and L\'evy distribution ($\alpha=\frac{1}{2}$, $\beta=1$). To deal with the lack of explicit densities for other cases the series expansions were established, see \cite{Zolotariev:1986} and \cite[Chapt. XVII, Sect. 6]{Feller:1971cr}. For $\sigma=1, \beta=0$ there is a following series expansion of the density function of $X$. For $\alpha\in(0,1)$ $$f_{X}(x)=\frac{1}{\pi}\sum_{n\geq1}\frac{(-1)^{n+1}}{n!}\Gamma(n\alpha+1)\sin\left(\frac{n\alpha}{2}\right)\frac{1}{x^{n\alpha+1}}.$$
For $\alpha\in(1,2]$ (\cite[Chapt. IV, Sect. 1]{Zolotarev})
    $$f_{X}(x)=\frac{1}{\alpha\pi}\sum_{n\geq1}\frac{(-1)^{n}}{2n!}\Gamma\left(\frac{2n+1}{\alpha}\right)x^{2n}.$$
Tail asymptotics of $\alpha$-stable distributions are well-known \cite[Property 1.2.15]{SamorodnitskyTaqqu:1994}. For $\alpha\in(0,2)$ 

$$\lim_{y\rightarrow\ +\infty}y^{\alpha}\P(X\geq y)=C_{\alpha}\frac{1+\beta}{2}\sigma^{\alpha}$$and
$$\lim_{y\rightarrow\ -\infty}|y|^{\alpha}\P(X\leq y)=C_{\alpha}\frac{1-\beta}{2}\sigma^{\alpha},$$where
$$C_{\alpha}=\left(\int_{0}^{\infty}x^{\alpha}\sin x dx\right)^{-1}=\frac{1}{\alpha\Gamma(-\alpha)\cos\left(\frac{(2-\alpha)\pi}{2}\right)}.$$
Observe that
$$C_{\alpha}=
\begin{cases}
1+o(1)&\text{if}\quad \alpha\rightarrow0^{+}\\
\frac{2}{\pi}(1+o(1))&\text{if}\quad\alpha\rightarrow1\\
(2-\alpha)(1+o(1))&\text{if}\quad\alpha\rightarrow2^{-}.
\end{cases}$$
For $\beta=-1$, in which case $\lim_{y\rightarrow\ +\infty}y^{\alpha}\P(X\geq y)=0$, so the rate of convergence of $\P(X\geq y)$ to $0$ is faster than $\frac{1}{y^{\alpha}}$. It is known \cite[(1.2.11)]{SamorodnitskyTaqqu:1994} that the right rate of convergence  for $\alpha\in(1,2)$ is given by $$
\frac{1+o(1)}{\sqrt{2\alpha\pi(\alpha-1})}\left(\frac{|y|}{\kappa_{\alpha}}\right)^{-\frac{\alpha}{2(\alpha-1)}}\exp\left(-(\alpha-1)\left(\frac{|y|}{\kappa_{\alpha}}\right)^{\frac{-\alpha}{\alpha-1}}\right),\\$$
where$\quad \kappa_{\alpha}=\frac{\alpha\sigma}{\cos\left(\frac{(2-\alpha)\pi}{2}\right)}$
and exactly the same for the left tails in the case of $\beta=1$. Recall that for $\alpha<1$ and $\beta=1$, $\P(X<0)=0$.\\

\noindent There is a rich literature on numerical calculation of stable densities and distribution
functions, see for example \cite{Nolan:1997} and references therein.
In this article we are interested in 'qualitative' behavior of tails of
symmetric and totally asymmetric $\alpha$-stable distributions.
More precisely, we are interested in the description of these
tails in terms of functions of the parameters of a distribution up
to universal constants.\\
\noindent Let $X$ be an $\alpha$-stable random variable.
As presented above, the asymptotic behavior of $\P\left(|X|>t\right)$ as $t\ra+\ns$ is
fully understood, but the value of the tail $\P\left(X>t\right)$
for moderate values of $t$ seems to be not well investigated. The study of densities of $\alpha$-stable distributions goes back to P\'olya \cite{Polya} as well as Blumenthal and Getoor \cite{Blumenthal}. Upper bounds for densities of the multidimensional $\alpha$-stable random variables were given in the work of Watanabe \cite{Watanabe:2007}. The classic work by W.E. Pruitt \cite{Pruitt1981} provides estimates for the tails of suprema of L\'{e}vy processes. The idea of truncating the characteristic function used both in \cite{Pruitt1981} and \cite{Watanabe:2007} is also applied in this work.  Some of the results presented here can be related to much more general work of T. Grzywny et.al \cite{Grzywny:2014} where estimates for densities were delivered toghether with explicit constants \cite{Grzywny':2014}, which are however of rather intricate form.  Also, upper bounds for $\beta\neq0$ can be found in \cite{Sztonyk}, while lower bounds for $|\beta|\neq1$ in \cite{Grzywny}.\\
\noindent The value of the results presented here lies mainly in the transparency of constants in estimates, which, as believed, were not explicitly presented so far. Also the approach based on elementary techniques might be of independent interest especially since it outlines the nature of alpha-stable variables whose tails are determined by the analysis of heavy-tailed jumps. The main novelty to the results in \cite{Grzywny:2014} is that we also consider strictly asymmetric case ($\beta=1$). Finally, calculations we provide for $\alpha$ close to $2$ allow to establish the order of boundary value at which the tail of $\alpha$-stable random variable alters from behaving like a Gaussian and starts to resemble a tail of Pareto distribution (see Remark \ref{alter}). \\

\noindent \textbf{Acknowledgements.} The authors wish to thank the anonymous referee whose notes helped to fix some bugs in the first draft of the manuscript.

\section{Methods}
Our approach is based on the analysis of the series representation of $\alpha$-stable random variables as well as their characteristic functions. First, we present  a classic series representation (see \cite{SamorodnitskyTaqqu:1994} section 1.4). \\
\noindent
Let $(\tau_{i})_{i\geq1}$ be a sequence of arrival times of Poissonian process with parameter 1 i.e. $\tau_{i}=\Gamma_{1}+\dots+\Gamma_{i}$, where the sequence  $(\Gamma_{k})_{k\geq1}$ is i.i.d. and for $u\geq0$, $\mathbb{P}(\Gamma_{k}\geq u)=e^{-u}$, then 
\begin{itemize}
    \item ($\alpha\in(0,1)$, $\beta=1$) $X\overset{d}{=}\sum_{i=1}^{\infty}(\alpha\tau_{i})^{-\frac{1}{\alpha}},$
    \item ($\alpha\in(0,1)$ and $\alpha\in(1,2)$, $\beta=0$) $X\overset{d}{=}\left(\frac{\alpha}{2}\right)^{-\frac{1}{\alpha}}\sum_{i=1}^{\infty}\varepsilon_{i}\tau_{i}^{-\frac{1}{\alpha}}$, (where $\varepsilon_{i}$ are independent Rademacher random variables),
    \item ($\alpha\in(1,2)$, $\beta=1$) $X\overset{d}{=}c_{\alpha}\sum_{i=1}^{\infty}\left(\tau_{i}^{-\frac{1}{\alpha}}-a_{i}\right)$, for an $\alpha$-dependent costant $c_\alpha$ and compensating terms $a_i$ given by $a_i=\frac{\alpha-1}{\alpha}\left(i^{\frac{\alpha-1}{\alpha}}-(i-1)^{\frac{\alpha-1}{\alpha}}\right)$. 
\end{itemize}
Series representations are particularly useful for simulations (see e.g. \cite{Rosinskisym}). Also, it is worth mentioning that a more general class of infinitely divisible processes admit a similar representation to the above known as Rosinski's representation \cite{Rosinskirep}.
Working with this representation turns out to be efficient when estimating tails of both symmetric and asymmetric $\alpha$-stable random variables for $\alpha\in(0,1)$. The proof of convergence of the above series can be found in \cite{SamorodnitskyTaqqu:1994}. To verify that the above series representations are right one needs to simply calculate the characteristic function of $X$ in each case and check that it obeys the definition (\ref{char}). The following two lemmas might serve as a tool in it and also will be helpful in further calculations. 
\begin{lem}\label{lem:expect}
Consider a Borel function $f:\mathbb{R}^{+}\rightarrow\mathbb{R}^{+}$ with $\int_{0}^{\infty} f(x)dx<\infty$. Then, $\mathbb{E}\sum_{i=1}^{\infty}f(\tau_{i})=\int_0^{\infty} f(x)dx.$
\end{lem}
\begin{proof}
It is a consequence of the fact that for each $i\geq 1$,   $\tau_{i}$ has the Erlang distribution i.e. its' density function is given by $\frac{x^{i-1}e^{-x}}{(i-1)!}$, where $x\geq0$. Since $f$ is non-negative and integrable we can put the summation outside the expectation. The result then follows easily. 
\end{proof}
The second lemma uses equivalence between Poissonian arrival times and Poissonian point processes and we omit the proof of it.

\begin{lem}[{\cite[ Lemma 11.3.3]{Tal}}]\label{lem:char}
For any $a>0$ and a continuous function $f:\mathbb{R}^{+}\rightarrow\mathbb{C}$ it holds that $\mathbb{E}\prod_{\tau_{i}<a}f(\tau_{i})=\exp\left(-\int_{0}^{a}(1-f(x))dx\right).$
\end{lem}
With the above properties the calculations of the characteristic function for the asymmetric case and $\alpha\in(0,1)$ are straightforward, while in the symmetric case it suffices to notice that the characteristic function can be expressed as
\begin{eqnarray*}
\E e^{itX} & = & \exp\left(\int_{0}^{\infty}\left(e^{itx}-1-itx\right)\frac{dx}{x^{\alpha+1}}+\int_{-\infty}^{0}\left(e^{itx}-1-itx\right)\frac{dx}{|x|^{\alpha+1}}\right)\\
& = & \exp\left(\int_{0}^{\infty}\left(e^{itx}-1-itx+e^{-itx}-1+itx\right)\frac{dx}{x^{\alpha+1}}\right)\\
& = &\exp\left(2\int_{0}^{\infty}(\cos(tx)-1)\frac{dx}{x^{\alpha+1}}\right).
\end{eqnarray*}
The main trick used when dealing with totally asymmetric case for $\alpha\in(0,1)$ is conditioning the series $\sum_{i=1}^{\infty}(\alpha\tau_{i})^{-\frac{1}{\alpha}}$ on the first term. To this end we observe that $\sum_{i=1}^{\infty}(\alpha\tau_{i})^{-\frac{1}{\alpha}}$ can be rewritten as 
$(\alpha\tau_{1})^{-\frac{1}{\alpha}}+\sum_{i=1}^{\infty}\alpha^{-\frac{1}{\alpha}}(\tau_{1}+\tilde{\tau_{i}})^{-\frac{1}{\alpha}},$
where for $i \geq 1$ we define 
\begin{equation}\label{tilde}
\tilde{\tau}_i = \tau_{i+1} - \tau_1.
\end{equation}
We notice that $\tilde{\tau}_i \overset{d}{=} \tau_i$ and $\tilde{\tau}_i$, $i \ge 1$, are independent from $\tau_i$.
For $x>0$ define the series 
$$S(x)=\sum_{i=1}^{\infty}\alpha^{-\frac{1}{\alpha}}(x+\tilde{\tau_i})^{-\frac{1}{\alpha}}.$$It is well-defined. Notice that $S(x)$ is decreasing. With the use of Lemma \ref{lem:char} we calculate moments of $S(x)$. 

\begin{lem}{\label{mom}}
The moment generating function of $S(x)$ is given by $$\Lambda_{S(x)}(\lambda)=\E\exp(\lambda S(x))=\exp(-f(\lambda,x)),\quad \lambda\geq0,$$
where
\begin{equation}
    f(\lambda,x)=\int_0^\infty1-\exp\left(\alpha^{-1/\alpha}\lambda(x+y)^{-1/\alpha}\right)dy.
\end{equation}
\end{lem}
\begin{proof}
Let $a>0$. Then, by Lemma \ref{lem:char},
\begin{eqnarray*}
\E\exp\left(\lambda\sum_{\tilde{\tau_i}<a}\alpha^{-1/\alpha}(x+\tilde{\tau_i})^{-1/\alpha}\right)=\E\prod_{\tilde{\tau_i}<a}\exp\left(\lambda\alpha^{-1/\alpha}(x+\tilde{\tau_i})^{-1/\alpha}\right)\\
=\exp\left(-\int_0^a1-\exp\left(\alpha^{-1/\alpha}\lambda(x+y)^{-1/\alpha}\right)dy\right).
\end{eqnarray*}
Passing on both sides to the limit as $a\rightarrow\infty$ is allowed since $S(x)$ is a convergent series and the integral on the right-hand side stays finite. To see this we use inequality $1-e^{u}\geq-2u$ for small, positive $u$. Consider sufficiently large constant $y_0$ and the quantity $I_{y_0}=\int_0^{y_0}1-\exp\left(\alpha^{-1/\alpha}\lambda(x+y)^{-1/\alpha}\right)dy$, which is bounded. Then,  
\begin{eqnarray*}
f(\lambda,x)&=&I_{y_0}+\int_{y_0}^\infty1-\exp\left(\alpha^{-1/\alpha}\lambda(x+y)^{-1/\alpha}\right)dy\\
&\geq& I_{y_0} -2\lambda\alpha^{-1/\alpha}\int_{y_0}^\infty(x+y)^{-1/\alpha}dy\\
&=&I_{y_0}-2\lambda\frac{(\alpha (x+y_0))^{1-1/\alpha}}{1-\alpha}>-\infty.
\end{eqnarray*}
\end{proof}
Therefore we can calculate any moment of $S(x)$. In particular, we have the following result. Obviously it could be also deduced from Lemma \ref{lem:expect}.
\begin{lem}{\label{war}}
With the above notation we have for $\alpha\in(0,1)$ that $$\E(S(x))=\frac{(\alpha x)^{1-\frac{1}{\alpha}}}{1-\alpha}
\qquad\mbox{and}\qquad \Var (S(x))=\frac{(\alpha x)^{1-\frac{2}{\alpha}}}{2-\alpha}.$$
\end{lem}
\begin{proof}
Fix $x>0$. Let's notice that $f(0,x)=0$ and use the notation $\frac{\partial f}{\partial\lambda}=f'$, $\frac{\partial^2 f}{\partial\lambda^2}=f''$. Simple calculation yields 

$$\mathbb{E}(S(x))=-f'|_{\lambda=0}=\int_0^\infty\alpha^{-1/\alpha}(x+y)^{-1/\alpha}dy=\frac{(\alpha x)^{1-\frac{1}{\alpha}}}{1-\alpha}.$$
Moreover,
$$-f''|_{\lambda=0}=\int_0^\infty\alpha^{-2/\alpha}(x+y)^{-2/\alpha}dy=\frac{(\alpha x)^{1-\frac{2}{\alpha}}}{2-\alpha}$$
and
$$\mathbb{E}\left(S(x)^2\right)=-f''|_{\lambda=0}+(f')^2|_{\lambda=0},$$
so $\Var(S(x))=\E(S(x)^2)-(\E(S(x)))^2=-f''|_{\lambda=0}.$
\end{proof}
Now, we outline tools which we use for analysing tails by the means of characteristic functions. For any random variable $Z$ we denote by $\varphi_{Z}(t)$ its' characteristic function. First, we recall elementary but very useful result which we apply in the symmetric case for all $\alpha\in(0,1)$. 
\begin{lem}[{\cite[Lemma 5.1]{Kallenberg}}]\label{Kal}
For any random variable $Z$ on $\mathbb{R}$ we have
$$\mathbb{P}(|Z|>y)\leq\frac{y}{2}\int_{-\frac{2}{y}}^{\frac{2}{y}}(1-\varphi_Z (t))dt.$$
\end{lem}
Next, we introduce the idea of truncating the characteristic function which will be applied for the case of $\alpha\in(1,2)$.
Let's start with considering totally asymmetric random variable with the characteristic function
\begin{equation} \label{char_funct_a12}
\E\exp\left(itX\right)=\exp\left(\int_{0}^{+\ns}\left(e^{itx}-1-itx\right)\frac{dx}{x^{\alpha+1}}\right).
\end{equation}
Opposite to the asymmetric case when $\alpha \in (0,1),$ the support of the distribution of a random variable $X$ with the characteristic function given by (\ref{char_funct_a12}) is the whole real line. Thus, we need upper and lower estimates for both right and left tails. The method is to split $X$ into the sum $X=X_{1}+X^{1}$ such that  
\begin{equation} \label{char_f_X_1}
\varphi_{X_{1}}(t)=\exp\left(\int_{0}^{1}\left(e^{itx}-1-itx\right)\frac{dx}{x^{\alpha+1}}\right),
\end{equation}
\begin{equation} \label{char_f_X^1}
\varphi_{X^{1}}(t)=\exp\left(\int_{1}^{+\ns}\left(e^{itx}-1-itx\right)\frac{dx}{x^{\alpha+1}}\right).
\end{equation}
It is easy to calculate that
\[
\int_{1}^{+\ns}(e^{itx}-1-itx)\frac{dx}{x^{\alpha+1}}=\int_{1}^{+\ns}(e^{itx}-1)\frac{dx}{x^{\alpha+1}}-\frac{it}{\alpha-1}
\]
thus the characteristic function of $X^{1}$ can be expressed as
\[
\varphi_{X^{1}}(t)=\exp\left(\frac{1}{\alpha}(\varphi_{Y}(t)-1)-\frac{it}{\alpha-1}\right),
\]
where the random variable $Y$ has the density function given by $\frac{\alpha}{x^{\alpha+1}}\mathbbm{1}_{(1,+\ns)}(x)$. This means that $X^{1}+\frac{1}{\alpha-1}$ has compoud Poisson distribution i.e.  
$$
X^{1}+\frac{1}{\alpha-1}=\sum_{k=1}^{N}Y_{k},
$$
where $N\sim \operatorname{Poisson}(\frac{1}{\alpha})$ while $Y_{k}$'s  are independent random variables all distributed as $Y$ and independent from $N$. \\
Similarly, for the symmetric $\alpha$-stable random variable $X$ with $\alpha\in(1,2)$ with the characteristic function
\begin{equation} \label{char_funct_s12}
\varphi_{X}(t)=\exp\left(\int_{-\ns}^{+\ns}\left(e^{itx}-1\right)\frac{dx}{|x|^{\alpha+1}}\right)
\end{equation}
we use the split $X=\tilde{X_{1}}+\tilde{X^{1}}$, where 
\begin{equation} \label{char_f_X_1s}
\varphi_{\tilde{X_{1}}}(t)=\exp\left(\int_{-1}^{1}\left(e^{itx}-1\right)\frac{dx}{|x|^{\alpha+1}}\right),
\end{equation}
\begin{equation} \label{char_f_X^1s}
\varphi_{\tilde{X}^{1}}(t)=\exp\left(\int_{\mathbb{R}\backslash [-1,1]}\left(e^{itx}-1\right)\frac{dx}{|x|^{\alpha+1}}\right).
\end{equation}
Analogously to the asymmetric case we observe that 
\[
\varphi_{\tilde{X}^{1}}(t)=\exp\left(\frac{2}{\alpha}(\varphi_{\tilde{Y}}(t)-1)\right),
\]
where the random variable $\tilde{Y}$ has a density function given by $\frac{\alpha}{2|x|^{\alpha+1}}\mathbbm{1}_{\mathbb{R}\backslash [-1,1]}(x)$. So, again $\tilde{X}^{1}$ is compound Poisson given by $\tilde{X}^{1}=\sum_{k=1}^{\tilde{N}}\tilde{Y}_{k}$, where $\tilde{N}\sim \operatorname{Poisson}(\frac{2}{\alpha})$ and $\tilde{Y}_{k}$'s are independent all distributed as $\tilde{Y}$ and independent of $\tilde{N}$.
\section{Results for $\alpha\in(0,1)$}

\subsection{Totally asymmetric case}
Wel now present results for the totally asymmetric $\alpha$-stable random variable $X$ with the characteristic function given by
$$\E\exp(itX)=\exp\left(\int_{0}^{\ns}\left(e^{itx}-1\right)\frac{1}{x^{\alpha+1}}dx\right)$$
with the series representation $X\overset{d}{=}\sum_{i=1}^{\infty}(\alpha\tau_i )^{-1/\alpha}$.
\begin{theorem}{\label{asymgor}}
Let $\alpha\in(0,1)$ and $y\geq 1$. For totally asymmetric $\alpha$- stable random variable $X$ we have the following tail estimate
\begin{equation}
    \mathbb{P}\left(X\ge\frac{1}{1-\alpha}+3y\right)\leq\frac{2}{\alpha y^{\alpha}}.
\end{equation}
Moreover, for $y\ge 1$ and $\theta\in(0,1)$ we have
\begin{equation}
\mathbb{P}\rbr{X\ge \frac{\theta}{1-\alpha}+y}\ge\frac{2}{3}(1-\theta)^{2}\frac{1}{1+\alpha y^{\alpha}}.
\end{equation}
\end{theorem}
\begin{proof}
From Lemma \ref{war} it follows that $$\mathbb{E}S\left(\frac{1}{\alpha y^{\alpha}}\right)=\frac{y^{1-\alpha}}{1-\alpha} \;\;\mbox{and}\;\;\Var\left(S\left(\frac{1}{\alpha y^{\alpha}}\right)\right)=\frac{y^{2-\alpha}}{2-\alpha}.$$ Now, 
\begin{eqnarray*}
&\mathbb{P}&\left(X\ge\frac{1}{1-\alpha}+3y\right)=\int_0^{\infty}e^{-x}\mathbb{P}\left((\alpha x)^{-1/\alpha}+S(x)>3y+\frac{1}{1-\alpha}\right)dx\\
&\leq &  \int_{0}^{1/(\alpha y^{\alpha})}e^{-x}dx +  \int_{{1/(\alpha y^{\alpha})}}^{\infty}e^{-x}\mathbb{P}\left(S(x)\geq 3y-(\alpha x)^{-1/\alpha}+\frac{1}{1-\alpha}\right)dx\\
&\le & 1-e^{-1/(\alpha y^{\alpha})}+\int_{{1/(\alpha y^{\alpha})}}^{\infty}e^{-x}\P\left(S(x)\ge2y+\frac{1}{1-\alpha}\right)dx\\
&\le& \frac{1}{(\alpha y^{\alpha})}+\P\left(S\left(\frac{1}{\alpha y^{\alpha}}\right)\ge 2y+\frac{1}{1-\alpha}\right)\int_{{1/(\alpha y^{\alpha})}}^{\infty}e^{-x}dx\\
&\leq& \frac{1}{(\alpha y^{\alpha})}+\P\left(S\left(\frac{1}{\alpha y^{\alpha}}\right)\ge y+\mathbb{E}S\left(\frac{1}{\alpha y^{\alpha}}\right)\right)\int_{{1/(\alpha y^{\alpha})}}^{\infty}e^{-x}dx\\
 &\le& \frac{1}{\alpha y^{\alpha}}+\frac{\Var\left(S\left(\frac{1}{\alpha y^{\alpha}}\right)\right)}{y^{2}}e^{-\frac{1}{\alpha y^{\alpha}}}=\frac{1}{\alpha y^{\alpha}}+\frac{1}{(2-\alpha)y^{\alpha}}e^{-\frac{1}{\alpha y^{\alpha}}}
   \le  \frac{2}{\alpha y^{\alpha}},
\end{eqnarray*}
where in the third inequality we used elementary inequality $1-e^{-u}\le u$. Also, if $x>\frac{1}{\alpha y^{\alpha}}$, then $(\alpha x)^{-\frac{1}{\alpha}}<y$. Next, we used the fact that for $y\ge 1$ we have $y+\frac{y^{1-\alpha}}{1-\alpha}\le2y+\frac{1}{1-\alpha}$ and then we applied Chebyshev's inequality.\smallskip

\noindent
We now turn to the lower bound. We use the decomposition $X\overset{d}{=}\rbr{\frac{1}{\alpha\tau_{1}}}^{\frac{1}{\alpha}}+S\left(\tau_{1}\right)$. Note that if $x\leq 1/(\alpha y^{\alpha})$ then $(\alpha x)^{-\frac{1}{\alpha}}\geq y$ and hence
$$
\mathbb{P}\left((\alpha\tau_1)^{-\frac{1}{\alpha}}+S(\tau_1)\geq y+\frac{\theta}{1-\alpha}\right)\geq\int_{0}^{\frac{1}{\alpha y^{\alpha}}}e^{-x}\mathbb{P}\left(S(x)\geq\frac{\theta}{1-\alpha}\right)dx.
$$
Moreover, since $y\geq1$ and $x\leq 1/(\alpha y^{\alpha})$ we have $S(x)\geq S\left(\frac{1}{\alpha}\right)$ so by the Paley-Zygmund inequality and Lemma \ref{war} we obtain

\begin{align*}
&\int_{0}^{\frac{1}{\alpha y^{\alpha}}}e^{-x}\mathbb{P} \left(  S(x)  \geq \frac{\theta}{1-\alpha}\right)dx  \ge  \int_{0}^{\frac{1}{\alpha y^{\alpha}}}e^{-x}\mathbb{P}\left(S\left(\frac{1}{\alpha}\right)\geq\frac{\theta}{1-\alpha}\right)dx\\
&= \int_{0}^{\frac{1}{\alpha y^{\alpha}}}e^{-x}\mathbb{P}\left(S\left(\frac{1}{\alpha}\right) \ge \theta\E S\left(\frac{1}{\alpha}\right)\right)dx
\ge\int_{0}^{\frac{1}{\alpha y^{\alpha}}}e^{-x}(1-\theta)^{2}\frac{1}{1+\frac{(1-\alpha)^{2}}{(2-\alpha)}}dx\\
&\ge   \frac{2}{3}(1-\theta)^{2}\left(1-\exp\left(-\frac{1}{\alpha y^{\alpha}}\right)\right),
\end{align*}
where in the last line we used that $\frac{1}{1+\frac{(1-\alpha)^{2}}{(2-\alpha)}}\ge\frac{2}{3}$.  The conclusion follows by the inequality $1-e^{-1/u}>\frac{1}{1+u}$ for $u>0$.

\end{proof}

\subsection{Symmetric case}
The lower bound for the symmetric case coincides for $\alpha\in(0,1)$ and $\alpha\in(1,2)$. However, as explained in the last section, further analysis is provided to reveal the Gaussian nature of tails in the latter case. 
\begin{theorem}\label{symm}
Let $X$ be a symmetric $\alpha$-stable random variable. Let $y>0$. We have the following estimate for the tail of $X$. For $\alpha\in(0,1)$ 
\begin{equation}
\P(X\geq y)\leq \frac{4}{\alpha y^{\alpha}}
\end{equation}
and for $\alpha\in(0,1)\cup(1,2)$
\begin{equation}
\P(X\geq y)\ge \frac{1}{2}\frac{1}{2+\alpha y^{\alpha}}.
\end{equation}
\end{theorem}

\begin{proof}
In order to apply Lemma \ref{Kal} we need a lower estimate for the characteristic function. To this end notice that

\begin{align*}
\varphi_{X}(t)&=\exp\left(\int_{-\infty}^{\infty}e^{itx}-1\frac{dx}{|x|^{\alpha+1}}\right)=\exp\left(2\int_0^{\infty}\cos(tx)-1\frac{dx}{x^{\alpha+1}}\right)\\
&=\exp\left(-2|t|^\alpha\int_{0}^{\infty}(1-\cos z)\frac{dz}{z^{\alpha+1}}\right)\geq\exp\left(-|t|^\alpha\cdot2^{-\alpha}\frac{8}{\alpha(2-\alpha)}\right),
\end{align*}
where in the last inequality we used
\begin{align*}
\int_{0}^{\infty}\left(1-\cos z\right)\frac{dz}{z^{\alpha+1}}\leq\int_0^2 \frac{z^2}{2}\frac{dz}{z^{\alpha+1}}+\int_2^{\infty}2\frac{dz}{z^{\alpha+1}}=2^{-\alpha}\cdot\frac{4}{\alpha(2-\alpha)}.
\end{align*}
Denote $C_\alpha=\frac{8}{\alpha(2-\alpha)}$. Then
\begin{align*}
\mathbb{P}(|X|>y)&\leq\frac{y}{2}\int_{-\frac{2}{y}}^{\frac{2}{y}}1-\exp\left(-\left(\frac{|t|}{2}\right)^\alpha\frac{8}{\alpha(2-\alpha)}\right)dt\\
&= 2y\int_{0}^{\frac{1}{y}}1-\exp(-C_\alpha s^{\alpha})ds\\
&\leq\frac{2C_\alpha}{(1+\alpha)}\frac{1}{y^{\alpha}}\leq\frac{1}{\alpha y^\alpha}\frac{16}{(1+\alpha)(2-\alpha)}\leq\frac{8}{\alpha y^\alpha},
\end{align*}
where in the second inequality we used $1-e^{-u}\le u$ and in the last the fact that $\alpha<1$.

\noindent For the lower bound we again condition on the first arrival time and use $(\tilde{\tau}_i)_{i\geq1}$ defined in (\ref{tilde}). By the symmetry we have
\begin{eqnarray*}
\P(X\ge y)&=& \P\left(\varepsilon_1 \left(\frac{\alpha}{2} \tau_1\right)^{-1/\alpha}+\left(\frac{\alpha}{2}\right)^{-1/\alpha}\sum_{i=1}^{\infty}\varepsilon_i(\tilde{\tau}_i+\tau_1)^{-1/\alpha}\ge y\right)\\
&\ge& \P\left(\varepsilon_1 \left(\frac{\alpha}{2} \tau_1\right)^{-1/\alpha}\ge y \;\;\mbox{and}\;\;\left(\frac{\alpha}{2}\right)^{-1/\alpha}\sum_{i=1}^{\infty}\varepsilon_i(\tilde{\tau}_i+\tau_1)^{-1/\alpha}\ge 0\right)\\
&=&\frac{1}{2}\P\left(\varepsilon_1 \left(\frac{\alpha}{2} \tau_1\right)^{-1/\alpha}\ge y\right)=\frac{1}{4}\int_0^{\frac{2}{\alpha y^{\alpha}}}e^{-x}dx\ge\frac{1}{2}\frac{1}{2+\alpha y^{\alpha}},
\end{eqnarray*}
where we used the inequality $1-e^{-1/u}>\frac{1}{1+u}$ for $u>0$.
\end{proof}

\section{Results for $\alpha \in (1,2)$}
\subsection{Totally asymmetric case}
We now consider random variable $X$ with the characterisitc function (\ref{char_funct_a12}) and we use the split $X=X_1+X^1$, where characterictic functions of $X_1$ and $X^1$ are given by (\ref{char_f_X_1}) and (\ref{char_f_X^1}).

\begin{lem} \label{lem1}
For $y \ge 1$ one has the following lower bound
\begin{equation} \label{low_b_12_assym}
\P\left(X^{1}\ge y-\frac{1}{\alpha-1}\right)\ge e^{-1/\alpha}\frac{1}{\alpha}\frac{1}{y^{\alpha}} \ge \frac{1}{2\sqrt{e}} \frac{1}{y^{\alpha}}
\end{equation}
and the following upper bound
\begin{equation}\label{up_b_12_assym}
\P\left(X^{1}\ge y-\frac{1}{\alpha-1}\right) \le \left(e^{-1/\alpha}\sum_{k=1}^{+\ns}\frac{k^{\alpha+1}}{\alpha^{k}k!}\right)\frac{1}{y^{\alpha}} \le \frac{2}{y^{\alpha}}
.
\end{equation}
\end{lem}
\begin{proof}
We notice that 
\[
\P\left(X^{1}\ge y-\frac{1}{\alpha-1}\right)\ge\P\left(N=1\right)\P\left(Y_{1}\ge y\right)=e^{-1/\alpha}\frac{1}{\alpha}\frac{1}{y^{\alpha}}.
\]
and $e^{-1/\alpha}/{\alpha} \ge 1/\rbr{2\sqrt{e}},$ since, by simple calculus, the function $\alpha \mapsto e^{-1/\alpha}/{\alpha}$ is decreasing on the interval $[1,2].$ 

\noindent On the other hand, whenever $\sum_{n=1}^{N}Y_{n}\ge t$ and $N=k$
we have that at least for one $i=1,2,\ldots,k,$ $Y_{i}\ge y/k$ which
occurs with probability no greater than $\sum_{i=1}^{k}\P\left(Y_{i}\ge y/k\right),$
thus 
\begin{align*}
&\P\left(X^{1}\ge t-\frac{1}{\alpha-1}\right)\\
&  \le  \sum_{k=1}^{+\ns}\P\left(N=k\right)\left(\sum_{i=1}^{k}\P\left(Y_{i}\ge y/k\right)\right) \le  e^{-1/\alpha}\sum_{k=1}^{+\ns}\frac{1}{\alpha^{k}k!}k\left(\frac{y}{k}\right)^{-\alpha}\\
 & =  \left(e^{-1/\alpha}\sum_{k=1}^{+\ns}\frac{k^{\alpha+1}}{\alpha^{k}k!}\right)\frac{1}{y^{\alpha}} 
  \le  \frac{2}{y^{\alpha}},
\end{align*}
where we used the fact that for each $k$ function $e^{-1/\alpha}\frac{k^{\alpha+1}}{\alpha^{k}k!}$ is decreasing in $\alpha$ therefore we plug in $\alpha=1$ and notice that $\sum_{k=1}^{\infty}\frac{k^2}{k!}=2e$.
\end{proof}

\noindent Now we proceed to analyse $X_{1}$ which is a much more delicate task. 
\begin{lem} \label{lem2}
For $0\le y\le\frac{1}{2-\alpha}$ one has
\begin{eqnarray*}
\P\left(X_{1}\ge y\right) & \le & e^{\frac{1}{4}}e^{-\frac{1}{2}\left(2-\alpha\right)y^{2}}
\end{eqnarray*}
and for $0\le y\le\frac{2}{2-\alpha}$
\begin{eqnarray*}
\P\left(X_{1}\le-y\right) & \le & e^{\frac{4}{3}}e^{-\frac{1}{2}\left(2-\alpha\right)y^{2}}.
\end{eqnarray*}
\end{lem}
\begin{proof}
We calculate 
\begin{eqnarray*}
\E\exp\left(tX_{1}\right) & = & \exp\left(\int_{0}^{1}e^{tx}-1-tx\frac{dx}{x^{\alpha+1}}\right)\\
 & = & \exp\left(\frac{1}{2}\frac{t^{2}}{2-\alpha}+\int_{0}^{1}e^{tx}-1-tx-\frac{1}{2}t^{2}x^{2}\frac{dx}{x^{\alpha+1}}\right).
\end{eqnarray*}
We estimate the integrated term using the following observation. Since
$tx\ge0,$ we have
\begin{align*}
e^{tx}-1-tx-\frac{1}{2}t^{2}x^{2}&=\frac{1}{3!}t^3x^3\sum_{k=0}^{\infty}\frac{3!}{(k+3)!}t^kx^k\\
&=\frac{1}{3!}t^3x^3\sum_{k=0}^{\infty}\frac{3!k!}{(k+3)!}\frac{1}{k!}t^kx^k\\
&\le\frac{1}{6}t^3x^3\left(\frac{3}{4}+\frac{1}{4}\sum_{k=0}^{\infty}\frac{1}{k!}t^kx^k\right)\\
&=\frac{1}{6}t^3x^3\left(\frac{3}{4}+\frac{1}{4}e^{tx}\right).
\end{align*}
For $t\ge0$ we estimate 
\begin{eqnarray}
\E\exp\left(tX_{1}\right) & \le & \exp\left(\frac{1}{2}\frac{t^{2}}{2-\alpha}+\left(\frac{1}{8}t^3+\frac{1}{24}t^3e^t\right)\int_{0}^{1}x^{3}\frac{dx}{x^{\alpha+1}}\right)\nonumber \\
 & = & \exp\left(\frac{1}{2}\frac{t^{2}}{2-\alpha}+\frac{1}{3-\alpha}\left(\frac{1}{8}t^3+\frac{1}{24}t^3e^t\right)\right).\label{eq:exp_mom1}
\end{eqnarray}
Now, for $0\le y\le\frac{1}{2-\alpha},$ taking $t_{y}=\left(2-\alpha\right)y$
we get $t_{y}\le1$. By Chebyschev's inequality and (\ref{eq:exp_mom1}) we get
\begin{align*}
\P\left(X_{1}\ge y\right)\le\E\exp\left(t_{y}X_{1}\right)e^{-t_{y}y} &\le\exp\left(-\frac{1}{2}\left(2-\alpha\right)y^{2}+\frac{1}{8} + \frac{1}{24}e\right)\\
&\le e^{1/4} e^{-\frac{1}{2}(2-\alpha) y^2}.
\end{align*}
Similarly, since for $tx\le0,$ $\left|e^{tx}-1-tx-\frac{1}{2}t^{2}x^{2}\right|\le\frac{1}{3!}\left|t^{3}x^{3}\right|,$
for $t\le0$ we have
\begin{eqnarray}\label{lem1.}
\E\exp\left(tX_{1}\right) & \le & \exp\left(\frac{1}{2}\frac{t^{2}}{2-\alpha}+\frac{\left|t^{3}\right|}{6}\int_{0}^{1}x^{3}\frac{dx}{x^{\alpha+1}}\right)\nonumber \\
 & = & \exp\left(\frac{1}{2}\frac{t^{2}}{2-\alpha}+\frac{1}{6}\frac{\left|t^{3}\right|}{3-\alpha}\right).
\end{eqnarray}
Again, for $0\le y\le\frac{2}{2-\alpha},$ taking $t_{y}=-\left(2-\alpha\right)y$
by Chebyschev's inequality we get 
\[
\P\left(X_{1}\le-y\right)\le\E\exp\left(-t_{y}X_{1}\right)e^{t_{y}y}\le\exp\left(-\frac{1}{2}\left(2-\alpha\right)y^{2}+\frac{8}{6}\right).
\]
\end{proof}
\
For lower bounds we use the Paley-Zygmund inequality. 
\begin{lem} \label{lem3}
For $\alpha\in(7/4,2)$ and $y\in\left[\frac{2}{\sqrt{2-\alpha}},\frac{1}{2-\alpha}\right]$ one has
\begin{equation} \label{low_12_assym_x1}
\P\left(X_{1}\ge\frac{1}{4}y\right) \ge 10^{-2} e^{-\left(2-\alpha\right)y^{2}} 
\end{equation}
while for $\alpha\in(1,2)$ and $y\in\left[\frac{2}{\sqrt{2-\alpha}}, \frac{2}{2-\alpha}\right]$ one has
\begin{equation} \label{low_12_assym_x2}
\P\left(X_{1}\le-\frac{1}{24}y\right) \ge e\cdot 10^{-3}e^{-\left(2-\alpha\right)y^{2}}.
\end{equation}
\end{lem}
\begin{proof}
Since for $tx\ge0,$ $e^{tx}-1-tx-\frac{1}{2}t^{2}x^{2}\ge0,$
for $t\ge0$ we estimate 
\begin{eqnarray}
\E\exp\left(tX_{1}\right) & \ge & \exp\left(\frac{1}{2}\frac{t^{2}}{2-\alpha}\right).\label{eq:exp_mom2}
\end{eqnarray}
Next, notice that for $y\ge\frac{2}{\sqrt{2-\alpha}}$ we have $\frac{1}{y}\le\frac{2-\alpha}{4}y$
so 
\[
\frac{1}{2}y-\frac{1}{2-\alpha}\frac{1}{y}\ge\frac{1}{2}y-\frac{1}{2-\alpha}\frac{2-\alpha}{4}y=\frac{1}{4}y
\]
and for $t_{y}=\left(2-\alpha\right)y$ and $\lambda=\frac{1}{e}$ by (\ref{eq:exp_mom2})
we have 
\begin{eqnarray*}
\frac{1}{t_{y}}\ln\left(\lambda\E\exp\left(t_{y}X_{1}\right)\right) & \ge & \frac{1}{t_{y}}\ln\left(\lambda\exp\left(\frac{1}{2}\frac{t_{y}^{2}}{2-\alpha}\right)\right)\\
 & = & \frac{1}{2}\frac{t_{y}}{2-\alpha}+\frac{\ln\lambda}{t_{y}}=\frac{1}{2}y-\frac{1}{2-\alpha}\frac{1}{y}\\
 & \ge & \frac{1}{4}y.
\end{eqnarray*}
This, together with the Paley-Zygmund inequality, (\ref{eq:exp_mom2}) and (\ref{eq:exp_mom1}) (notice that for $y \le \frac{1}{2-\alpha},$  $t_y \le 1$) yields (\ref{low_12_assym_x1}):

\begin{eqnarray*}
\P\left(X_{1}\ge\frac{1}{4}y\right) & \ge & \P\left(X_{1}\ge\frac{1}{t_{y}}\ln\left(\lambda\E\exp\left(t_{y}X_{1}\right)\right)\right)\\
 & = & \P\left(\exp\left(t_{y}X_{1}\right)\ge\lambda\E\exp\left(t_{y}X_{1}\right)\right)\\
 & \ge & \left(1-\frac{1}{e}\right)^{2}\frac{\left(\E\exp\left(t_{y}X_{1}\right)\right)^{2}}{\E\exp\left(2t_{y}X_{1}\right)}\\
 & \ge & \left(1-\frac{1}{e}\right)^{2}\frac{\exp\left(\frac{t_{y}^{2}}{2-\alpha}\right)}{\exp\left(\frac{2t_{y}^{2}}{2-\alpha}+{t_y^3+\frac{1}{3}t_y^3e^{2t_y}}\right)}\\
 & \ge & \left(1-\frac{1}{e}\right)^2e^{-(1+\frac{1}{3}e^2)}e^{-\left(2-\alpha\right)y^{2}}\\
 & \ge & 10^{-2}e^{-\left(2-\alpha\right)y^{2}}.
 \end{eqnarray*}
For negative tails we use estimate $e^{tx}-1-tx-\frac{1}{2}t^{2}x^{2}\ge\frac{1}{3!}t^{3}x^{3}$
for $tx\le0,$ which for $t\le0$ yields \begin{equation}
\E\exp\left(tX_{1}\right)\ge\exp\left(\frac{1}{2}\frac{t^{2}}{2-\alpha}-\frac{1}{6}\left|t^{3}\right|\right).\label{eq:exp_mom3}
\end{equation}
Next, notice that for $\frac{2}{\sqrt{2-\alpha}}\le y\le\frac{2}{2-\alpha}$
we have $\frac{1}{y}\le\frac{2-\alpha}{4}y$ and $\left(2-\alpha\right)^{2}y^{2}\le4\le2y$,
so 
\[
\frac{1}{2}y-\frac{1}{2-\alpha}\frac{1}{2y}-\frac{1}{6}\left(2-\alpha\right)^{2}y^{2}\ge\frac{1}{2}y-\frac{1}{2-\alpha}\frac{2-\alpha}{8}y-\frac{1}{3}y=\frac{1}{24}y.
\]
From this for $t_{y}=-\left(2-\alpha\right)y$ and $\lambda=1/\sqrt{e},$
by (\ref{eq:exp_mom3}) we have 
\begin{eqnarray*}
\frac{1}{\left|t_{y}\right|}\ln\left(\lambda\E\exp\left(t_{y}X_{1}\right)\right) & \ge & \frac{1}{\left|t_{y}\right|}\ln\left(\lambda\exp\left(\frac{1}{2}\frac{t_{y}^{2}}{2-\alpha}-\frac{1}{6}\left|t_{y}^{3}\right|\right)\right)\\
 & = & \frac{1}{2}\frac{\left|t_{y}\right|}{2-\alpha}+\frac{\ln\lambda}{\left|t_{y}\right|}-\frac{1}{6}\left|t_{y}^{2}\right|\\
 &=&\frac{1}{2}y-\frac{1}{2-\alpha}\frac{1}{2y}-\frac{1}{6}\left(2-\alpha\right)^{2}y^{2}\\
 & \ge & \frac{1}{24}y.
\end{eqnarray*}
This, together with the Paley-Zygmund inequality yields 
\begin{eqnarray*}
\P\left(X_{1}\le-\frac{1}{24}y\right) & \ge & \P\left(X_{1}\le\frac{1}{t_{y}}\ln\left(\lambda\E\exp\left(t_{y}X_{1}\right)\right)\right)\\
 & = & \P\left(\exp\left(t_{y}X_{1}\right)\ge\lambda\E\exp\left(t_{y}X_{1}\right)\right)\\
 & \ge & \left(1-\frac{1}{\sqrt{e}}\right)^{2}\frac{\left(\E\exp\left(t_{y}X_{1}\right)\right)^{2}}{\E\exp\left(2t_{y}X_{1}\right)}\\
 & \ge & \left(1-\frac{1}{\sqrt{e}}\right)^{2}\frac{\exp\left(\frac{t_{y}^{2}}{2-\alpha}-\frac{8}{3}\right)}{\exp\left(\frac{2t_{y}^{2}}{2-\alpha}+\frac{8}{6}\right)}\\
 & = & e^{2\ln\left(\sqrt{e}-1\right)-5}e^{-\left(2-\alpha\right)y^{2}}\\
 &\ge & e\cdot 10^{-3}e^{-\left(2-\alpha\right)y^{2}}.
\end{eqnarray*}
\end{proof}
To complete the picture we estimate $\P\left(X_{1}\le-y\right)$ in the case $y\geq\frac{2}{2-\alpha}.$
\begin{lem} \label{lem4}
For $y\ge\frac{2}{2-\alpha}$ one has
\begin{align*}
\P\left(X_{1} \le -y\right)\le  \exp\left(-\frac{\left(\frac{1}{2}\left(y+\frac{1}{\alpha-1}\right)\right)^{\alpha/\left(\alpha-1\right)}}{\left(\frac{1}{2-\alpha}+\frac{1}{\alpha-1}\right)^{1/\left(\alpha-1\right)}}\right)
\end{align*}
and
\begin{align*}
  & \P\left(X_{1}\le-\left(\frac{1}{e}-\frac{1}{4}\right)y\right)\\
  & \ge  \left(1-\frac{1}{\sqrt{e}}\right)^{2}\exp\left(-\frac{\left(\sqrt{4-\frac{2}{e}}\left(y+\frac{1}{\alpha-1}\right)\right)^{\alpha/\left(\alpha-1\right)}}{\left(\frac{1}{2-\alpha}+\frac{1}{\alpha-1}\right)^{1/\left(\alpha-1\right)}}\right).
\end{align*}
\end{lem}
\begin{proof}
For $t<-1$ first we split 
\[
\int_{0}^{1}e^{tx}-1-tx\frac{dx}{x^{\alpha+1}}=\int_{0}^{1/\left|t\right|}e^{tx}-1-tx\frac{dx}{x^{\alpha+1}}+\int_{1/\left|t\right|}^{1}e^{tx}-1-tx\frac{dx}{x^{\alpha+1}}.
\]
For $t<-1$ and $0\le x\le\frac{1}{\left|t\right|}$ we calculate $e^{tx}-1-tx\le\frac{1}{2}t^{2}x^{2}\le t^{2}x^{2}$
and we get 
\begin{eqnarray*}
\int_{0}^{1/\left|t\right|}e^{tx}-1-tx\frac{dx}{x^{\alpha+1}} & \le & \int_{0}^{1/\left|t\right|}t^{2}x^{2}\frac{dx}{x^{\alpha+1}}=\frac{1}{2-\alpha}\left|t\right|^{\alpha}.
\end{eqnarray*}
Next, for $x>\frac{1}{\left|t\right|}$ we bound $e^{tx}-1-tx\le-tx=\left|t\right|x$
and get 
\[
\int_{1/\left|t\right|}^{1}e^{tx}-1-tx\frac{dx}{x^{\alpha+1}}\le\left|t\right|\int_{1/\left|t\right|}^{1}x\frac{dx}{x^{\alpha+1}}=\frac{1}{\alpha-1}\left(\left|t\right|^{\alpha}-\left|t\right|\right).
\]
Finally, we arrive at 
\[
\int_{0}^{1}e^{tx}-1-tx\frac{dx}{x^{\alpha+1}}\le\left(\frac{1}{2-\alpha}+\frac{1}{\alpha-1}\right)\left|t\right|^{\alpha}-\frac{1}{\alpha-1}\left|t\right|
\]
which yields that for $t<-1$ 
\begin{equation}\label{mom3}
\E\exp\left(tX_{1}\right)\le\exp\left(\left(\frac{1}{2-\alpha}+\frac{1}{\alpha-1}\right)\left|t\right|^{\alpha}-\frac{1}{\alpha-1}\left|t\right|\right).
\end{equation}
Let $y\geq\frac{2}{2-\alpha}$ and $t_{y}<-1$ be such that 
\begin{equation}
\alpha\left(\frac{1}{2-\alpha}+\frac{1}{\alpha-1}\right)\left|t_{y}\right|^{\alpha-1}=y+\frac{1}{\alpha-1}.\label{eq:t_y_def}
\end{equation}
We estimate 
\begin{align*}
 & \P\left(X_{1}<-y\right)\le\E\exp\left(-t_{y}X\right)e^{t_{y}y}\\
 & \le  \exp\left(\left(\frac{1}{2-\alpha}+\frac{1}{\alpha-1}\right)\left|t_{y}\right|^{\alpha}-\frac{1}{\alpha-1}\left|t_{y}\right|-y\left|t_{y}\right|\right)\\
 & =  \exp\left(-\left(\alpha-1\right)\left(\frac{1}{2-\alpha}+\frac{1}{\alpha-1}\right)\left|t_{y}\right|^{\alpha}\right)\\
 & =  \exp\left(-\frac{\left(\left(y+\frac{1}{\alpha-1}\right)\left(\alpha-1\right)^{\left(\alpha-1\right)/\alpha}/\alpha\right)^{\alpha/\left(\alpha-1\right)}}{\left(\frac{1}{2-\alpha}+\frac{1}{\alpha-1}\right)^{1/\left(\alpha-1\right)}}\right)\\
 & \le  \exp\left(-\frac{\left(\frac{1}{2}\left(y+\frac{1}{\alpha-1}\right)\right)^{\alpha/\left(\alpha-1\right)}}{\left(\frac{1}{2-\alpha}+\frac{1}{\alpha-1}\right)^{1/\left(\alpha-1\right)}}\right),
\end{align*}
where we used the estimate $\inf_{\alpha \in(1,2)}\frac{\left(\alpha-1\right)^{\frac{\alpha-1}{\alpha}}}{\alpha}= \frac{1}{2}.$
On the other hand, for $t<-1$ and $0\le x\le\frac{1}{\left|t\right|}$ we
have $e^{tx}-1-tx\ge\frac{1}{e}t^{2}x^{2}$ and we get 
\begin{eqnarray*}
\int_{0}^{1/\left|t\right|}e^{tx}-1-tx\frac{dx}{x^{\alpha+1}} & \ge & \int_{0}^{1/\left|t\right|}\frac{t^{2}x^{2}}{e}\frac{dx}{x^{\alpha+1}}=\frac{1}{e}\frac{1}{2-\alpha}\left|t\right|^{\alpha}.
\end{eqnarray*}
Similarly, for $x>\frac{1}{\left|t\right|}$ we bound $e^{tx}-1-tx\ge-\frac{1}{e}tx=\frac{1}{e}\left|t\right|x.$
So,
\[
\int_{1/\left|t\right|}^{1}e^{tx}-1-tx\frac{dx}{x^{\alpha+1}}\ge\frac{1}{e}\left|t\right|\int_{1/\left|t\right|}^{1}x\frac{dx}{x^{\alpha+1}}=\frac{1}{e}\frac{1}{\alpha-1}\left(\left|t\right|^{\alpha}-\left|t\right|\right).
\]
Finally, we arrive at the estimate 
\begin{equation}\label{mom4}
\E\exp\left(tX_{1}\right)\ge\exp\left(\frac{1}{e}\left(\frac{1}{2-\alpha}+\frac{1}{\alpha-1}\right)\left|t\right|^{\alpha}-\frac{1}{e}\frac{1}{\alpha-1}\left|t\right|\right)
\end{equation}
which for $\tilde{t}_{y}<-1$ satisfying 
\begin{equation}
\left(\frac{1}{2-\alpha}+\frac{1}{\alpha-1}\right)\left|\tilde{t}_{y}\right|^{\alpha-1}=\frac{1}{\alpha-1}+y\label{eq:t_tilde}
\end{equation}
which is equivalent to
\[
\left(\frac{1}{2-\alpha}+\frac{1}{\alpha-1}\right)\left|\tilde{t}_{y}\right|^{\alpha}-\frac{1}{\alpha-1}\left|\tilde{t}_{y}\right|=\left|\tilde{t}_{y}\right|y
\]
and for $\lambda=\frac{1}{\sqrt{e}}$ yields 
\begin{eqnarray}
\frac{1}{\left|\tilde{t}_{y} \right|}\ln\left(\lambda\E\exp\left(\tilde{t}_{y} X \right)\right)\nonumber 
& \ge & \frac{1}{\left|\tilde{t}_{y}\right|}\left(\ln\left(\lambda\right)+\frac{1}{e}\left(\frac{1}{2-\alpha}+\frac{1}{\alpha-1}\right)\left|\tilde{t}_{y}\right|^{\alpha}-\frac{1}{e}\frac{1}{\alpha-1}\left|\tilde{t}_{y}\right|\right)\nonumber \\
& =&\frac{1}{\left|\tilde{t}_{y}\right|}\left(\ln\left(\lambda\right)+\frac{1}{e}\left|\tilde{t}_{y}\right|y\right)=\frac{1}{e}y-\frac{1}{2\left|\tilde{t}_{y}\right|}.\label{eq:estim_estim}
\end{eqnarray}
To estimate $\frac{1}{\left|\tilde{t}_{y}\right|}$ let us notice that from
(\ref{eq:t_tilde}) for $y\ge\frac{2}{2-\alpha}\ge2$ we have
\[
\left|\tilde{t}_{y}\right|\ge\left|\tilde{t}_{y}\right|^{\alpha-1}=\frac{\frac{1}{\alpha-1}+y}{\frac{1}{2-\alpha}+\frac{1}{\alpha-1}}\ge1\ge\frac{2}{y}
\]
which together with (\ref{eq:estim_estim}) yields
$$
\frac{1}{\left|\tilde{t}_{y}\right|}\ln\left(\lambda\E\exp\left(\tilde{t}_{y}X\right)\right)\ge\frac{1}{e}y-\frac{1}{2\left|\tilde{t}_{y}\right|}\ge\left(\frac{1}{e}-\frac{1}{4}\right)y.
$$
Finally, using the just obtained estimate, the Paley-Zygmund inequality, (\ref{mom3}) and (\ref{mom4}) we arrive at 
\begin{align*}
  & \P\left(X_{1}\le-\left(\frac{1}{e}-\frac{1}{4}\right)y\right)\\
 & \ge  \P\left(X_{1}\le-\frac{1}{\left|\tilde{t}_{y}\right|}\ln\left(\lambda\E\exp\left(\tilde{t}_{y}X_{1}\right)\right)\right)\\
 & =  \P\left(\exp\left(\tilde{t}_{y}X\right)\ge\lambda\E\exp\left(\tilde{t}_{y}X_{1}\right)\right)\\
 & \ge  \left(1-\frac{1}{\sqrt{e}}\right)^{2}\frac{\left(\E\exp\left(\tilde{t}_{y}X_{1}\right)\right)^{2}}{\exp\left(2\tilde{t}_{y}X_{1}\right)}\\
 & \ge  \left(1-\frac{1}{\sqrt{e}}\right)^{2}\frac{\exp\left(\frac{2}{e}\left(\frac{1}{2-\alpha}+\frac{1}{\alpha-1}\right)\left|\tilde{t}_{y}\right|^{\alpha}-\frac{2}{e}\frac{1}{\alpha-1}\left|\tilde{t}_{y}\right|\right)}{\exp\left(\left(\frac{1}{2-\alpha}+\frac{1}{\alpha-1}\right)\left|2\tilde{t}_{y}\right|^{\alpha}-\frac{1}{\alpha-1}\left|2\tilde{t}_{y}\right|\right)}\\
 & \ge  \left(1-\frac{1}{\sqrt{e}}\right)^{2}\exp\left(-\left(4-\frac{2}{e}\right)\left(\frac{1}{2-\alpha}+\frac{1}{\alpha-1}\right)\left|\tilde{t}_{y}\right|^{\alpha}\right)\\
 & =  \left(1-\frac{1}{\sqrt{e}}\right)^{2}\exp\left(-\frac{\left(4-\frac{2}{e}\right)\left(y+\frac{1}{\alpha-1}\right)^{\alpha/\left(\alpha-1\right)}}{\left(\frac{1}{2-\alpha}+\frac{1}{\alpha-1}\right)^{1/\left(\alpha-1\right)}}\right)\\
 & \ge  \left(1-\frac{1}{\sqrt{e}}\right)^{2}\exp\left(-\frac{\left(\sqrt{4-\frac{2}{e}}\left(y+\frac{1}{\alpha-1}\right)\right)^{\alpha/\left(\alpha-1\right)}}{\left(\frac{1}{2-\alpha}+\frac{1}{\alpha-1}\right)^{1/\left(\alpha-1\right)}}\right).
\end{align*}
\end{proof}
As an easy consequence of Lemmas \ref{lem1}, \ref{lem2} and \ref{lem3} we have the following theorem.
\begin{theorem} \label{main_assym_12}
Let $X$ be a strictly asymmetric $\alpha$-stable random variable, with the characteristic function (\ref{char_funct_a12}). For any $\alpha \in (7/4, 2)$ and $y \in \sbr{\frac{2}{\sqrt{2 - \alpha}}, \frac{1}{2 - \alpha}}$ one has the following estimates
\begin{equation} \label{upp_b_skew_12}
\P\rbr{X \ge 2 y - \frac{1}{\alpha - 1}} \le \frac{2}{e} \frac{1}{y^{\alpha}} + e^{\frac{1}{4}}e^{-\frac{1}{2}(2-\alpha)y^2},
\end{equation}
\begin{align} \label{low_b_skew_12}
\P\rbr{X \ge  \frac{1}{4}y - \frac{1}{\alpha - 1}} & \ge \frac{1}{400\sqrt{e}}\left(30\frac{1}{y^{\alpha}}+e^{-(2-\alpha) y^2}\right);
\end{align}
while for $\alpha \in (1,2)$ and $y \ge \frac{1}{2 - \alpha}$ one has
\begin{equation} \label{upp_b_skew_12_big_y}
\P\rbr{X \ge 2 y - \frac{1}{\alpha - 1}} \le \frac{8}{y^{\alpha}},
\end{equation}
\begin{equation} \label{low_b_skew_12_big_y}
\P\rbr{X \ge  y - \frac{1}{\alpha - 1}} \ge 16\cdot10^{-3} \frac{1}{y^{\alpha}}.
\end{equation}
\end{theorem}
\begin{rem}
Notice that from (\ref{low_b_skew_12}) it follows that for $\alpha$ close to $2$ (in fact for $\alpha>7/4$) and $y = \frac{2}{\sqrt{2 - \alpha}}$ the probability $\P\rbr{X \ge  \frac{1}{4}y - \frac{1}{\alpha - 1}}$ is of order $O(1)$.  We seemingly lack the estimates for $\alpha \in (1, 7/4)$ but in this case $\frac{1}{{2 - \alpha}} = O(1)$ and from (\ref{low_b_skew_12_big_y}) it follows that for $\alpha \in (1, 7/4)$ the probability $\P\rbr{X \ge  y - \frac{1}{\alpha - 1}}$ is of order $O(1)$ even for $y = \frac{1}{{2 - \alpha}}$. 
\end{rem}
\begin{proof}
To prove (\ref{upp_b_skew_12}) we estimate
\[
\P\rbr{X \ge 2 y - \frac{1}{\alpha - 1}}  \le \P\rbr{X^1 \ge  y - \frac{1}{\alpha - 1}}  + \P\rbr{X_1 \ge  y } 
\]
and then use (\ref{up_b_12_assym}) and the upper bound for $\P\rbr{X_1 \ge  y } $ from Lemma \ref{lem2}.

\noindent To prove (\ref{low_b_skew_12}) we write for $y \in \sbr{\frac{2}{\sqrt{2-\alpha}}, \frac{1}{2-\alpha}}$
\[
\P\rbr{X \ge \frac{1}{4}y - \frac{1}{\alpha - 1}} \ge \P\rbr{X^1 \ge  \frac{5}{4} y - \frac{1}{\alpha - 1}}\P\rbr{X_1 \ge -y} 
\]
and then use (\ref{low_b_12_assym}) and Lemma  \ref{lem2} to obtain
\begin{align} 
 \P\rbr{X \ge  \frac{1}{4}y - \frac{1}{\alpha - 1}} & \ge \P\rbr{X^1 \ge  \frac{5}{4} y - \frac{1}{\alpha - 1}}\P\rbr{X_1 \ge -y} \nonumber \\
& \ge \frac{1}{2\sqrt{e}} \frac{4^{\alpha}}{5^{\alpha} y^{\alpha}}  \rbr{1- e^{\frac{4}{3}} e^{-\frac{1}{2}(2-\alpha)y^2} } \nonumber \\& \ge \frac{1}{2\sqrt{e}} \frac{16}{25 y^{\alpha}}  \rbr{1- e^{\frac{4}{3}} e^{-\frac{1}{2}(2-\alpha)\frac{4}{2-\alpha}} } \nonumber \\& \ge \frac{0.3}{2\sqrt{e}} \frac{1}{ y^{\alpha}} \label{f1}.
\end{align}
Next, for $y \in \sbr{\frac{2}{\sqrt{2-\alpha}}, \frac{1}{2-\alpha}}$ we also have 
\[
\P\rbr{X \ge  \frac{1}{4}y - \frac{1}{\alpha - 1}}\ge \P\rbr{X^1 \ge  1 - \frac{1}{\alpha - 1}}\P\rbr{X_1 \ge  \frac{1}{4}y} 
\]
which, together with (\ref{low_b_12_assym}) and (\ref{low_12_assym_x1}) gives
\begin{equation} \label{f2}
\P\rbr{X \ge  \frac{1}{4}y - \frac{1}{\alpha - 1}} \ge \frac{1}{2\sqrt{e}}10^{-2}e^{-(2-\alpha) y^2}.
\end{equation}
Summing corresponding sides of estimates (\ref{f1}) and (\ref{f2}) we get (\ref{low_b_skew_12}).

\noindent To prove (\ref{upp_b_skew_12_big_y}), we differentiate (\ref{char_f_X_1}) and get 
\[
\E X_1 = 0, \quad \E X_1^2 = \int_0^1 x^2 \frac{dx}{x^{\alpha +1}}=  \frac{1}{2-\alpha}
\]
and
\[
\E X_1^4 = 3 (\E X_1^2)^2 + \int_0^1 x^4 \frac{dx}{x^{\alpha +1}}=  \frac{3}{\rbr{2-\alpha}^2} + \frac{1}{4-\alpha}.
\]
From this we easily get for any $y>0$ the estimate 
\begin{equation} \label{fourth_mom_Ch}
\P \rbr{X_1 \ge y} \le \P \rbr{\left|X_1 \right| \ge y} \le \frac{\E X_1^4}{ y^4} = \frac{3}{(2-\alpha)^2 y^4} + \frac{1}{(4-\alpha) y^4}
\end{equation}
and since for $y \ge \frac{1}{2-\alpha}$
\begin{equation} \label{simp_but_imp}
\frac{1}{(4-\alpha) y^4} \le \frac{1}{(2-\alpha)^2 y^4} \le \frac{1}{y^2} \le \frac{1}{y^{\alpha}},
\end{equation}
using also (\ref{up_b_12_assym}), we obtain (\ref{upp_b_skew_12_big_y}):
\begin{eqnarray*}
 \P \rbr{X \ge 2y - \frac{1}{\alpha -1}} & \le & \P \rbr{X_1 \ge y} + \P \rbr{X^1 \ge y - \frac{1}{\alpha -1}}  \\
& \le & \frac{3}{y^{\alpha}} + \frac{1}{y^{\alpha}} 
+ \frac{5}{\sqrt{e}} \frac{1}{y^{\alpha}} \le \frac{8}{y^{\alpha}}. 
\end{eqnarray*}
To prove (\ref{low_b_skew_12_big_y})  for $y \ge \frac{1}{2-\alpha}$ we write
\[
\P\rbr{X \ge y - \frac{1}{\alpha - 1}} \ge \P\rbr{X^1 \ge  3 y - \frac{1}{\alpha - 1}}\P\rbr{X_1 \ge - 2 y} 
\]
and then use (\ref{low_b_12_assym}) and Lemma  \ref{lem2} to obtain
\begin{align*} 
 \P\rbr{X \ge y - \frac{1}{\alpha - 1}} & \ge \P\rbr{X^1 \ge  3 y - \frac{1}{\alpha - 1}}\P\rbr{X_1 \ge -2y} \nonumber \\
& \ge \frac{1}{2\sqrt{e}} \frac{1}{3^{\alpha} y^{\alpha}}  \P\rbr{X_1 \ge  - \frac{2}{2-\alpha}} \nonumber \\& \ge \frac{1}{2\sqrt{e}} \frac{1}{9 y^{\alpha}}  \rbr{1- e^{\frac{4}{3}} e^{-\frac{1}{2}(2-\alpha)\frac{4}{(2-\alpha)^2}} } \nonumber \\& \ge 16\cdot 10^{-3} \frac{1}{ y^{\alpha}}.
\end{align*}
\end{proof}
\begin{rem}\label{alter}
For $\delta \in (0, \frac{1}{e})$ the equation $\delta \cdot y = \ln y $
has exactly two solutions $1 < y_1<e<y_2$, and the larger one satisfies 
\[
\frac{1}{\delta} \ln \frac{1}{\delta} < y_2 < \frac{2}{\delta} \ln \frac{1}{\delta}.
\]
From this we get that for $\alpha \approx 2,$ the term containing $\frac{1}{y^{\alpha}}$ in (\ref{upp_b_skew_12}) and (\ref{low_b_skew_12})
starts to dominate the term containing $\exp\rbr{-\kappa (2-\alpha) y^2},$ $\kappa \in \cbr{1/2, 1}$, already for 
\[
y = O\rbr{\sqrt{\frac{1}{2-\alpha} \ln \frac{1}{2-\alpha}}}.
\] 
\end{rem}
Finally, to complete the picture, we analyse the decay of left tails of $X$.
\begin{theorem}
Let $X$ be a strictly asymmetric $\alpha$-stable random variable, $\alpha \in (1,2),$ with the characteristic function (\ref{char_funct_a12}). For any $y\in \sbr{\frac{2}{\sqrt{2 - \alpha}}, \frac{2}{\rbr{2 - \alpha}}}$ one has the following estimates
\begin{equation} \label{upp_b_skew_12_left}
\P\rbr{X \le -y - \frac{1}{\alpha - 1}} \le e^{\frac{4}{3}}e^{-\frac{1}{2}\left(2-\alpha\right)y^{2}},
\end{equation}
\begin{equation} \label{low_b_skew_12_left}
\P\rbr{X \le - \frac{1}{24}y - \frac{1}{\alpha - 1}} \ge 10^{-3}e^{-\left(2-\alpha\right)y^{2}};
\end{equation}
while for $y \ge \frac{2}{2 - \alpha}$ one has
\begin{align}\label{upp_b_skew_12_big_y_left}
\P\left(X\le-y- \frac{1}{\alpha - 1} \right) 
\le  \exp\left(-\frac{\left(\frac{1}{2}\left(y+\frac{1}{\alpha-1}\right)\right)^{\alpha/\left(\alpha-1\right)}}{\left(\frac{1}{2-\alpha}+\frac{1}{\alpha-1}\right)^{1/\left(\alpha-1\right)}}\right) 
\end{align}
and
\begin{align}\label{low_b_skew_12_big_y_left}
  & \P\left(X\le-\left(\frac{1}{e}-\frac{1}{4}\right)y- \frac{1}{\alpha - 1}\right) \ge e^{-1}\exp\left(-\frac{\left(\sqrt{4-\frac{2}{e}}\left(y+\frac{1}{\alpha-1}\right)\right)^{\alpha/\left(\alpha-1\right)}}{\left(\frac{1}{2-\alpha}+\frac{1}{\alpha-1}\right)^{1/\left(\alpha-1\right)}}\right).
\end{align}
\end{theorem}
\begin{proof}
Estimate (\ref{upp_b_skew_12_left}) follows from Lemma \ref{lem2} and the fact that $X^1 \ge \frac{-1}{\alpha - 1}.$ 
Estimate (\ref{low_b_skew_12_left}) follows from Lemma \ref{lem3} and the fact that $\P \rbr{X^1 = \frac{-1}{\alpha - 1}} = e^{-1/\alpha} \ge \frac{1}{e}.$

\noindent Similarly, estimate (\ref{upp_b_skew_12_big_y_left}) follows from Lemma \ref{lem4} and the fact that $$X^1 \ge \frac{-1}{\alpha - 1}$$ while 
estimate (\ref{low_b_skew_12_big_y_left}) follows from the estimate $$\P \rbr{X^1 = \frac{-1}{\alpha - 1}} = e^{-1/\alpha} \ge \frac{1}{e}$$ and Lemma \ref{lem4}.
\end{proof}

\subsection{Symmetric case}
In this section we provide tail estimates for symmetric $\alpha$-stable random variables in the case when $\alpha\in(1,2)$. We follow two different approaches and as a consequence we obtain two types of bounds. The first method was already presented in Theorem \ref{symm}. Estimates obtained in this way hold on the whole real line, however do not capture an important property one might expect for $\alpha$ close to $2$, namely the Gaussian behavior of the tail which has already been presented in the asymmetric case. For this reason we show an analogous reasoning as in the previous section i.e. we need estimates of $\tilde{X}^1$ and $\tilde{X_1}$ with characteristic functions given by (\ref{char_f_X_1s}) and (\ref{char_f_X^1s}) respectively. To ease the notation we denote  $\tilde{X}^1$ by $X^1$ and $\tilde{X_1}$ by $X_1$. 
 Now, we proceed to the analysis of $X^1$ and $X_1$.

\begin{lem}
Let $y\geq 1$. We have the following estimates for tails of $X^1$.
\begin{equation}\label{X^1d}
\P(X^1\geq y)\geq \frac{1}{e}\frac{1}{y^\alpha}.
\end{equation}
and
\begin{equation}\label{X^1g}
\P(X^1\geq y)\leq \frac{1}{y^\alpha}\frac{1}{2}\sum_{k=1}^{\infty}\frac{e^{-\frac{2}{\alpha}}(\frac{2}{\alpha})^k k^{\alpha+1}}{k!}\leq \frac{10}{3}\frac{1}{y^\alpha}
\end{equation}
\end{lem}
\begin{proof}
Recall that $X^1=\sum_{k=1}^{N}Y_k$, where $\mathbb{P}(N=k)=\frac{e^{-\frac{2}{\alpha}}(\frac{2}{\alpha})^k}{k!}$ and each $Y_k$ has a density $\frac{\alpha}{2|x|^{\alpha+1}}\mathbbm{1}_{\mathbb{R}\backslash [-1,1]}(x)$.
Arguing in the same manner as in Lemma \ref{lem1} we obtain
$$\mathbb{P}(X^1\geq y)\geq\mathbb{P}(N=1)\mathbb{P}(Y_1\geq y)=\frac{2}{\alpha}e^{-\frac{2}{\alpha}}\frac{1}{y^\alpha}\geq\frac{1}{e}\frac{1}{y^\alpha},$$
since $\frac{2}{\alpha}e^{-\frac{2}{\alpha}}$ is increasing for $\alpha\in(1,2)$. For the upper bound
\begin{align*}
\mathbb{P}(X^1>y)&\leq\sum_{k=1}^{\infty}\mathbb{P}(N=k)\mathbb{P}\left(Y>\frac{y}{k}\right)\\
&\leq\sum_{k=1}^{\infty}\frac{e^{-\frac{2}{\alpha}}\left(\frac{2}{\alpha}\right)^k}{k!}\frac{k}{2}(\frac{y}{k})^{-\alpha}\\
&=\frac{1}{y^\alpha}\frac{1}{2}\sum_{k=1}^{\infty}\frac{e^{-\frac{2}{\alpha}}\left(\frac{2}{\alpha}\right)^k k^{\alpha+1}}{k!}\leq\frac{10}{3}\frac{1}{y^\alpha},
\end{align*}
where we estimate the function $\frac{e^{-\frac{2}{\alpha}}(\frac{2}{\alpha})^k k^{\alpha+1}}{k!}$ for $k=1,2,3$ by its values at $\alpha=2$ and for $k=4, 5,\dots$ by the values at $\alpha=1$.
\end{proof}
For both upper and lower bounds of tails of $X_1$ we need an estimate for its Laplace transform.

\begin{lem}
Let $X_1$ be a random variable with characteristic function given by (\ref{char_f_X^1s}). Then for $t\in\R$,
\begin{equation}\label{laplg}
\mathbb{E}(\exp(tX_1))\leq\exp\left(\frac{1}{24}t^4\left(\frac{14}{15}+\frac{1}{15}\cosh(t)\right)\right)\exp\left(\frac{1}{2-\alpha}t^2\right)
\end{equation}
\begin{equation}\label{lapld}
\mathbb{E}(\exp(tX_1))\geq\exp\bigg(\frac{1}{2-\alpha}t^2\bigg).
\end{equation}
\end{lem}
\begin{proof}
We simply calculate
\begin{align*}
\mathbb{E}(\exp(tX_1))&=\exp\bigg(\int_{-1}^1 e^{tx}-1-tx\frac{dx}{|x|^{\alpha+1}}\bigg)=\exp\bigg(2\int_0^1 \cosh(tx)-1\frac{dx}{x^{\alpha+1}}\bigg)\\
&=\exp\bigg(2\int_0^1\sum_{k=1}^{\infty}\frac{(t^2x^2)^{k}}{(2k)!}\frac{dx}{x^{\alpha+1}}\bigg)\\
&=\exp\bigg(\int_0^1 t^2 x^2\frac{dx}{x^{\alpha+1}}+2\int_0^1 \sum_{k=2}^{\infty}\frac{(t^2 x^2)^k}{(2k)!}\frac{dx}{x^{\alpha+1}}\bigg)\\
&=\exp\bigg(\frac{t^2}{2-\alpha}+\frac{2}{4!}\int_0^1 t^4 x^4\sum_{k=0}^{\infty}\frac{4!(2k)!(x^2 t^2)^k}{(2k+4)!(2k)!}\frac{dx}{x^{\alpha+1}}\bigg)\\
&\leq\exp\bigg(\frac{t^2}{2-\alpha}+\frac{2}{4!}\int_0^1 t^4 x^4\left(\frac{14}{15}+\frac{1}{15}\cosh(tx)\right)\frac{dx}{x^{\alpha+1}}\bigg)\\
&\leq \exp\left(\frac{ t^2}{2-\alpha}+\frac{2}{4!}t^4\left(\frac{14}{15}+\frac{1}{15}\cosh(t)\right)\frac{1}{4-\alpha}\right)\\
&\leq\exp\bigg(\frac{1}{24}t^4\left(\frac{14}{15}+\frac{1}{15}\cosh(t)\right)\bigg)\exp\bigg(\frac{1}{2-\alpha}t^2\bigg).
\end{align*}
The lower bound is obvious from the fourth line above.
\end{proof}
\noindent 
\begin{lem}
For $0\leq y\leq \frac{2}{2-\alpha}$ it holds that
\begin{equation}\label{X_1g}
\P(X_1\geq y)\leq e^{\frac{2}{45}}e^{-\frac{1}{4}(2-\alpha)y^2}
\end{equation}
and for $y\in[\frac{2}{\sqrt{2-\alpha}},\frac{2}{2-\alpha}]$
\begin{equation}\label{X_1d}
\P\left(X_1\geq \frac{\sqrt{2}}{4}y\right)\geq\frac{1}{137}e^{-(2-\alpha)y^2} 
\end{equation}
\end{lem}
\begin{proof}
Denote $C(t)=\frac{1}{24}t^4\big(\frac{14}{15}+\frac{1}{15}\cosh(t)\big)$. By Chebyshev's inequality and (\ref{laplg}) we get
$$\mathbb{P}(X_1>y)\leq\frac{\mathbb{E}(tX_1)}{\exp(ty)}\leq\exp(C(t))\exp\left(\frac{t^2}{2-\alpha}-ty\right).$$
Choose $t=\frac{2-\alpha}{2}y$, so $t\leq1$. Then, since $\cosh(1)\leq2$, $C(t)\leq\frac{2}{45}$ and we conclude that
$\P(X_1\geq y)\leq e^{\frac{2}{45}}e^{-\frac{1}{4}(2-\alpha)y^2}.$\\
To prove the lower bound we use Paley-Zygmund inequality in the following way. Let $\lambda\in(0,1)$, then
\begin{align*}
\mathbb{P}(\exp(tX_1)&\geq\lambda\mathbb{E}\exp(tX_1))
\geq (1-\lambda)^2\frac{(\mathbb{E}(t X_1))^2}{\mathbb{E}(2t X_1)}\\
&\geq  (1-\lambda)^2\frac{\exp(\frac{2t^2}{2-\alpha})}{\exp(C(2t))\exp(\frac{4t^2}{2-\alpha})}\\
&=(1-\lambda)^2\exp(-C(2t))\exp\left(-\frac{2t^2}{2-\alpha}\right).
\end{align*}
Choose $t=\frac{y(2-\alpha)}{\sqrt{2}}$, so $t\leq\sqrt{2}$ and $C(2t)\leq C(2\sqrt{2})$. Moreover, since $y\geq\frac{2}{\sqrt{2-\alpha}}$, we have for $\lambda=\frac{1}{e}$
\begin{align*}
\frac{1}{t}\ln(\lambda\mathbb{E}\exp(tX_1))&\geq\frac{1}{t}\ln\left(\lambda\exp\left(\frac{t^2}{2-\alpha}\right)\right)
=\frac{y}{\sqrt{2}}-\frac{\sqrt{2}}{(2-\alpha)y}\geq\frac{\sqrt{2}}{4}y,
\end{align*}
so, finally
\begin{align*}
\P\left(X_1\geq\frac{\sqrt{2}}{4}y\right)&\geq\mathbb{P}(\exp(tX_1)\geq\lambda\mathbb{E}\exp(tX_1))\\
&\geq(1-e^{-1})^2e^{-C(2\sqrt{2})}e^{-(2-\alpha)y^2}\geq\frac{1}{137}e^{-(2-\alpha)y^2}.
\end{align*}
\end{proof}
We summarize above results in the following.
\begin{theorem}
Let $X$ be a symmetric $\alpha$-stable random variable, $\alpha \in (1,2),$ with the characteristic function (\ref{char_funct_s12}). For any $y\in\sbr{\frac{2}{\sqrt{2-\alpha}}, \frac{2}{2-\alpha}}$ one has the following estimate
\begin{equation} \label{upp_b_sym_12}
\P\rbr{X \ge 2y} \le \frac{10}{3} \frac{1}{y^{\alpha}} + e^{\frac{2}{45}}e^{-\frac{1}{4}(2-\alpha)y^2}
\end{equation}
and
\begin{equation} \label{low_b_sym_12}
\P\rbr{X \ge \frac{\sqrt{2}}{4}y } \ge \frac{1}{4e}\frac{1}{y^{\alpha}} + \frac{1}{548}e^{-(2-\alpha)y^2};
\end{equation}
while for $y\geq\frac{2}{2-\alpha}$ one has
\begin{equation}\label{rosg1}
\P(X\ge 2y)\leq\frac{16}{3}\frac{1}{y^\alpha}
\end{equation}
and
\begin{equation}\label{rosd1}
\P(X\geq y)\ge \frac{1}{2}\frac{1}{2+\alpha y^{\alpha}}.
\end{equation}

\end{theorem}
\begin{proof}
We argue as in the proof of Theorem \ref{main_assym_12}. For the upper bound we simply apply (\ref{X^1g}) and (\ref{X_1g}) to get that
$$\P\rbr{X \ge 2y} \le\P(X_1\ge y)+\P(X^1\ge y)\le\frac{10}{3} \frac{1}{y^{\alpha}} + e^{\frac{2}{45}}e^{-\frac{1}{4}(2-\alpha)y^2}. $$
For the lower bound we use (\ref{X^1d}), (\ref{X_1d}) and symmetry of $X^1$ and $X_1$ to get 
$$\P\left(X\ge\frac{\sqrt{2}}{4}y\right)\ge\P(X\ge y)\ge \P(X^1\ge y)\P\left(X_1\ge 0 \right)\ge\frac{1}{e}\frac{1}{y^{\alpha}}\frac{1}{2}$$
and on the other hand
$$\P\left(X\ge\frac{\sqrt{2}}{4}y\right)\ge\P\left(X_1\ge\frac{\sqrt{2}}{4}y\right)\P(X^1\ge0)\ge\frac{1}{137}e^{-(2-\alpha)y^2}\frac{1}{2}.$$
Summing over both sides of the above inequalities yields (\ref{low_b_sym_12}).

\noindent To prove (\ref{rosg1}) we again proceed as in the proof of Theorem \ref{main_assym_12}, namely we differentiate \ref{char_f_X_1s} and get that
\[
\E X_1 = 0, \quad \E X_1^2 = 2\int_0^1 x^2 \frac{dx}{x^{\alpha +1}}=  \frac{2}{2-\alpha}
\]
and
\[
\E X_1^4 = 3 (\E X_1^2)^2 +2 \int_0^1 x^4 \frac{dx}{x^{\alpha +1}}=  \frac{12}{\rbr{2-\alpha}^2} + \frac{2}{4-\alpha}.
\]
By the same argument as for (\ref{fourth_mom_Ch}) and since $y\ge\frac{2}{2-\alpha}$ we get 
\begin{equation}
\P \rbr{X_1 \ge y}=\frac{1}{2}\P \rbr{\left|X_1 \right| \ge y} \le \frac{\E X_1^4}{ 2y^4} = \frac{1}{2}\rbr{\frac{12}{(2-\alpha)^2 y^4} + \frac{2}{(4-\alpha) y^4}}\le\frac{2}{y^\alpha}.
\end{equation}
Combining with (\ref{X^1g}) yields (\ref{rosg1}):
\begin{eqnarray*}
 \P \rbr{X \ge 2y } \le  \P \rbr{X_1 \ge y} + \P \rbr{X^1 \ge y }   \le  \frac{2}{y^{\alpha}} 
+ \frac{10}{3} \frac{1}{y^{\alpha}} \le \frac{16}{3}\frac{1}{y^{\alpha}}. 
\end{eqnarray*}
The estimate (\ref{rosd1}) was presented in the proof of Theorem \ref{symm}.
\end{proof}
\begin{rem}
Both remarks made after Theorem \ref{main_assym_12} apply also in this case. The fact that for $y=\frac{2}{\sqrt{2-\alpha}}$ the tail probability is of order $O(1)$ as well as that for $y$ of order  $O\rbr{\sqrt{\frac{1}{2-\alpha} \ln \frac{1}{2-\alpha}}}$ the term $\frac{1}{y^{\alpha}}$ (Pareto-like behaviour) starts to dominate the $\exp\rbr{-\kappa (2-\alpha) y^2},$ $\kappa \in \cbr{1/2, 1}$ term (Gaussian tail).
\end{rem}

\bibliographystyle{plain}

\begin{thebibliography}{1}


\bibitem{Blumenthal}
R.~M. Blumenthal and R.~K. Getoor.
\newblock Some theorems on stable processes.
\newblock {\em Trans. Amer. Math.Soc.} 95:263--273, 1960.

\bibitem{Grzywny:2014}
K. Bogdan, T. Grzywny and M. Ryznar. 
\newblock Density and tails of unimodal convolution semigroups.
\newblock {\em Journal of Functional Analysis},
266(6): 3543--3571, 2014.

\bibitem{Grzywny':2014}
K. Bogdan, T. Grzywny and M. Ryznar. 
\newblock Density and tails of unimodal convolution semigroups.
\newblock {\em arXiv:1305.0976v1}, 2013.

\bibitem{Grzywny}
T. Grzywny and K. Szczypkowski
\newblock Estimates of heat kernels of non-symmetric L\'evy processes
\newblock {\em arXiv:1710.07793v2}, 2020. 


\bibitem{Kallenberg}

O. Kallengerg.
\newblock {\em Foundations of Modern Probability}.
\newblock Springer-Verlag New York, 2002.

\bibitem{Feller:1971cr}
W. Feller.
\newblock {\em An introduction to probability theory and its applications.
  {V}ol. {II}.}
\newblock Second edition. John Wiley \& Sons Inc., New York, 1971.

\bibitem{Nolan:1997}
J.~P. Nolan.
\newblock Numerical calculation of stable densities and distribution functions.
\newblock {\em Comm. Statist. Stoch. Models}, 13:759--774, 1997.


\bibitem{Polya}
G. P\'olya
\newblock  On the zeros of an integral function represented by Fourier's integral
\newblock {\em Messenger of Math.} 52:185--188, 1923. 

\bibitem{Pruitt1981}
W.~E. Pruitt.
\newblock The growth of random walks and L\'evy processes.
\newblock {\em Annals of Probability}, 
9(6): 948--956, 1981.

\bibitem{Rosinskirep}
J. Rosi\'nski.
\newblock {On the series representation of infinitely divisible random vectors.}
\newblock {\em Annals of Probability},
18: 405--430, 1990.

\bibitem{Rosinskisym}
J. Rosi\'nski.
\newblock {Simulations of L\'evy Processes.}
\newblock {\em Encyclopedia of Statistics in Quality and Reliability: Computationally Intensive Methods and Simulation}.
\newblock John Wiley \& Sons, Ltd., 2008.

\bibitem{SamorodnitskyTaqqu:1994}
G. Samorodnitsky and M.~S. Taqqu.
\newblock {\em Stable Non-Gaussian Random Processes: Stochastic Models With
  Infinite Variance}.
\newblock Chapman and Hall, New York, 1994.

\bibitem{Sztonyk}
P. Sztonyk
\newblock Transition density estimates for jump L\'evy processes
\newblock {\em Stochastic Process. Appl.} 121(6): 1245--1265, 2011.
 
\bibitem{Tal}
M. Talagrand.  
\newblock {\em Upper and Lower Bounds for Stochastic Processes. Modern Methods and Classical Problems}, volume 60 of {\em A Series of Modern Surveys in Mathematics}, \newblock Springer-Verlag, New-York, 2014.

\bibitem{Zolotarev}
V. Uchaikin and V. Zolotarev 
\newblock{\em Chance and Stability. Stable Distributions and their Applications.}
\newblock  De Gruyter, Berlin, Boston, 2011.



\bibitem{Watanabe:2007}
T. Watanabe.
\newblock Asymptotic Estimates of Multi-Dimensional Stable Densities and Their Applications.
\newblock {\em Transactions of the American Mathematical Society},
359(6): 2851--2879, 2007.


\bibitem{Zolotariev:1986}
V.~M. Zolotarev.
\newblock {\em One-dimensional Stable Distributions}, volume~65 of {\em
  Translations of Mathematical Monographs}.
\newblock American Mathematical Society, Providence, Rhode Island, 1986.


\end{thebibliography}

\end{document}